\documentclass[a4paper,11pt, leqno]{article}

\usepackage[intlimits]{amsmath}
\usepackage{amssymb}
\usepackage{amsfonts}
\usepackage{mathrsfs}
\usepackage{amsthm}
\usepackage{colonequals,comment}
\usepackage[hidelinks]{hyperref} 
\usepackage[nameinlink]{cleveref} 
\usepackage{geometry}
\usepackage{mathtools}
\usepackage{xcolor}
\usepackage{relsize}
\usepackage{faktor}
\usepackage[T1]{fontenc}
\usepackage[english]{babel}
\usepackage{csquotes}
\usepackage[shortlabels]{enumitem}
\usepackage{setspace}
\usepackage[super]{nth}

\usepackage{tikz}
\usetikzlibrary{positioning,arrows} 
\usetikzlibrary{decorations.pathreplacing} 
\usetikzlibrary{intersections,shapes}
\usetikzlibrary{arrows.meta}
\usetikzlibrary{decorations.markings}
\usepackage{tikzsymbols} 
\usepackage{tikz-cd}
\tikzset{
  equal/.style={
    -,double,
    postaction={decorate,decoration={markings,mark=at position 0.5 with {\node[transform shape] {=};}}}
  }
}

\newcounter{listcounter}
\newcounter{deflistcounter}
\newcounter{equivcounter}

\newskip{\itemsepamount}
\newskip{\topsepamount}
\newenvironment{assertionlist}{%
  \begin{list}
    {\upshape (\arabic{listcounter})}
    {\setlength{\leftmargin}{18pt}
     \setlength{\rightmargin}{0pt}
     \setlength{\itemindent}{0pt}
     \setlength{\labelsep}{5pt}
     \setlength{\labelwidth}{13pt}
     \setlength{\listparindent}{\parindent}
     \setlength{\parsep}{0pt}
     \setlength{\itemsep}{\itemsepamount}
     \setlength{\topsep}{\topsepamount}
     \usecounter{listcounter}}}
  {\end{list}}

\newenvironment{definitionlist}{%
  \begin{list}
    {\upshape (\alph{deflistcounter})}
    {\setlength{\leftmargin}{18pt}
     \setlength{\rightmargin}{0pt}
     \setlength{\itemindent}{0pt}
     \setlength{\labelsep}{5pt}
     \setlength{\labelwidth}{13pt}
     \setlength{\listparindent}{\parindent}
     \setlength{\parsep}{0pt}
     \setlength{\itemsep}{\itemsepamount}
     \setlength{\topsep}{\topsepamount}
     \usecounter{deflistcounter}}}
  {\end{list}}

\newenvironment{equivlist}{%
  \begin{list}
    {\upshape (\roman{equivcounter})}
    {\setlength{\leftmargin}{18pt}
     \setlength{\rightmargin}{0pt}
     \setlength{\itemindent}{0pt}
     \setlength{\labelsep}{5pt}
     \setlength{\labelwidth}{13pt}
     \setlength{\listparindent}{\parindent}
     \setlength{\parsep}{0pt}
     \setlength{\itemsep}{\itemsepamount}
     \setlength{\topsep}{\topsepamount}
     \usecounter{equivcounter}}}
  {\end{list}}

\newenvironment{bulletlist}{%
  \begin{list}
    {\upshape \textbullet}
    {\setlength{\leftmargin}{18pt}
     \setlength{\rightmargin}{0pt}
     \setlength{\itemindent}{0pt}
     \setlength{\labelsep}{6pt}
     \setlength{\labelwidth}{12pt}
     \setlength{\listparindent}{\parindent}
     \setlength{\parsep}{0pt}
     \setlength{\itemsep}{\itemsepamount}
     \setlength{\topsep}{\topsepamount}}}
  {\end{list}}

% ...

\newcommand{\N}{\mathbb{N}}
\newcommand{\Z}{\mathbb{Z}}

\renewcommand{\O}{\mathcal{O}}
\newcommand{\Coker}{\mathrm{Coker}}
\newcommand{\Ker}{\mathrm{Ker}}
\renewcommand{\Im}{\mathrm{Im}}
\newcommand{\Hom}{\mathrm{Hom}}
\newcommand{\colim}{\mathrm{colim}}
\newcommand{\Spec}{\mathrm{Spec}\,}
\newcommand{\id}{\mathrm{id}}

\usepackage[
  backend=biber,
  style=numeric
]{biblatex}

\theoremstyle{plain}
\newtheorem{Theorem}{Theorem}[section]
\newtheorem{Proposition}[Theorem]{Proposition}
\newtheorem{Lemma}[Theorem]{Lemma}
\newtheorem{Corollary}[Theorem]{Corollary}

\newtheorem{customthm}{Theorem}
\newenvironment{cthm}[1]
  {\renewcommand\thecustomthm{#1}\customthm}
  {\endcustomthm}
\theoremstyle{definition}
\newtheorem{Definition}[Theorem]{Definition}
\newtheorem{customdef}{Definition}
\newenvironment{cdef}[1]
  {\renewcommand\thecustomdef{#1}\customdef}
  {\endcustomdef}
\newtheorem{Remark}[Theorem]{Remark}
\newtheorem{Reminder}[Theorem]{Reminder}

\newtheorem{Example}[Theorem]{Example}

  {\par\interskip}%

\numberwithin{equation}{subsection}

\author{\large Saskia Kern}
\date{\small \today} 
\title{\LARGE The crisp topology, a refinement of the fpqc topology}

\begin{document}

\pagenumbering{gobble}

    \maketitle
    \pagenumbering{arabic}
    \noindent{\scshape Abstract.\ }We introduce the crisp topology for schemes as a refinement of the fpqc topology. This Grothendieck topology uses the new notion of crisp morphisms, which generalise universal injectivity from ring homomorphisms to arbitrary morphisms of schemes. We study basic properties and demonstrate that this topology is well-behaved.

\section*{Introduction}\label{intro}

Pure morphisms of modules and commutative rings have been studied in France (cf. \cite{Laz69}, \cite{Oli70}) as a solution to the problem of finding some class of effective descent morphisms which is “weak” in the sense that it includes any other class of effective descent morphisms (cf. [Stacks, Tag 08WE]) -- it turns out that being a pure morphism is indeed equivalent to being an effective descent morphism for modules (cf. \cite[Theorem 3]{Oli70}). This leads to a close connection between pure and faithfully flat ring homomorphisms, as the latter are commonly studied effective descent morphisms. Moreover, Yves André and Luisa Fiorot have shown that on the category of affine schemes over a base commutative ring, the canonical and the effective descent topology coincide (cf. \cite[Theorem 3.2]{AF21}), further highlighting the importance of pure ring homomorphisms.  This paper aims to introduce crisp scheme morphisms, with which we try to solve the problem of generalising the notion of purity for ring homomorphisms to a naturally related concept for schemes. As usual, let all rings be commutative and with one. 

\begin{cdef}{I}[Definition~\ref{crisp-module-map}]
    Let $A$ be a ring and $M$ and $N$ be $A$-modules. An $A$-linear map $u\colon M\to N$ is called \textit{universally injective} or \textit{pure} if the morphism 
        \[u\otimes \id_P \colon M\otimes_A P\to N\otimes_A P\]
    is injective for all $A$-modules $P$, cf. \cite[p. 86]{Laz69}.
\end{cdef}

\begin{cdef}{II}[Definition~\ref{crisp-ring-hom}]
    Let $A$ and $B$ be rings. A ring homomorphism $\varphi\colon A\to B$ is called \textit{universally injective} or \textit{pure} if it is pure as a morphism in ($A$-Mod), cf. \cite[Definition 1.1]{Oli70}. 
\end{cdef}

Equivalently, a ring homomorphism $\varphi\colon A\to B$ is pure if for any $A$-algebra $R$, $R\to R\otimes_A B$ is injective (cf. \cite[Proposition 2.2]{AF21}). It is easy to see that this concept can be ``generalised'' to morphisms of affine schemes by requiring the associated ring homomorphism to be pure. However, as this doesn't extend to arbitrary scheme morphisms, an alternative approach has to be taken. In this context, we want to note that we will avoid the term ``universal injectivity'' from now on as for schemes, this has nothing to do with our sought-after generalisation.

Because purity for ring homomorphisms really lives in the categories of modules and algebras, an intuitive approach would perhaps be to adopt the definition of purity to quasi-coherent modules (almost) ad verbatim. Indeed, this has been done in papers by Picavet and Mesablishvili (cf. \cite{Pic13}, \cite{Mes04a}, \cite{Mes04b}): A quasi-compact morphism of schemes $f\colon X\to Y$ is called \textit{pure} if it is universally schematically dominant, i.e. if the morphism $\O_Y\to f_* \O_X$ is injective after arbitrary base change. Although this definition is meaningful as it coincides with purity of ring homomorphisms in the affine case (where we can consider ring algebras instead of modules, see Theorem \ref{crisp-equivalences} (ii)), it has drawbacks since we have to restrict ourselves to quasi-compact morphisms.

Since purity generalises faithful flatness for ring homomorphisms, another approach is to closely relate the definition of purity for schemes to the definition of faithful flatness for schemes by requiring a pure morphism to be surjective and pure on the stalks. This doesn't work either, since for affine schemes, this is not equivalent to the purity of the underlying ring homomorphisms, which will be demonstrated in Chapter 2 (cf. Example \ref{sch-examples}).

We thus propose a new generalisation which doesn't restrict itself to quasi-compact morphisms yet is compatible with the affine notion of purity:

\begin{cdef}{IV}[Definition~\ref{crisp-morphism}]
    A surjective morphism of quasi-separated schemes $f\colon X\to Y$ is called \textit{crisp}, if it is locally quasi-compact surjective and for every open affine subscheme $V\subseteq Y$, there is a quasi-compact open $U\subseteq f^{-1}(V)$ with $f(U) = V$ such that for every finite open affine covering $U_{1},\dots, U_{r}$ of $U$, the ring homomorphism corresponding to the morphism of affine schemes
            \[\bigsqcup_{i=1} ^r U_{i}\longrightarrow V\]
        is pure.
\end{cdef}

For other equivalent characterisations, we refer to Definition \ref{crisp-morphism}.

We will show  that the ring homomorphism associated with a morphism of affine schemes is pure if and only if the morphism of schemes is crisp, making this indeed a sensible generalisation. Moreover, any fpqc morphism of schemes is crisp, which preserves the relation to faithful flatness.

Being crisp is equivalent to being pure for quasi-compact morphisms of schemes. This means that the progress made in studying pure morphisms is also useful for crisp morphisms.

As the property ``pure'' has different meanings in different contexts, one might argue that since there now are crisp morphisms generalising Definition \ref{crisp-ring-hom}, the notion of purity could be dropped for rings and modules altogether for the sake of a well-defined nomenclature. Indeed, this paper proposes to call any morphism satisfying Definition \ref{crisp-module-map} or \ref{crisp-ring-hom} crisp instead of pure or universally injective. We will do so from now on.

A substantial part of this paper is dedicated to the study of crisp morphisms of schemes, leading to the following vital theorems for working with them (which will be proven in Chapter 2):

\begin{cthm}{V}[Theorem~\ref{permanence-crisp}]\label{ppintro}
    Crisp morphisms of schemes satisfy the following properties:
    \begin{assertionlist}
        \item Being crisp is stable under composition.
        \item Being crisp is stable under base change.
        \item Being crisp is local on the target.
    \end{assertionlist}
\end{cthm}

\begin{cthm}{VI}[Propositions~\ref{faithfully-flat-crisp}, \ref{pureqcqs-crisp} and \ref{crisp-subtrusive}]\label{crisp-context}
    Let $f\colon X\to Y$ be a morphism of schemes. 
    \begin{assertionlist}
        \item If $f$ is an fpqc covering, then it is crisp.
        \item If $f$ is quasi-compact, then it is crisp if and only if it is pure.
        \item If $f$ is crisp and quasi-compact, then it is universally subtrusive (and in particular universally submersive)\footnote{see Definitions \ref{submersive} and \ref{subtrusive} for details}.
    \end{assertionlist}
\end{cthm}

In addition to these basic statements, the paper introduces a crisp Grothendieck topology on the category of schemes over a base scheme $S$. 

\begin{cdef}{VII}[Definition~\ref{crisp-top}]
    The \textit{crisp} topology on $(\mathrm{Sch}/S)$ is the topology in which the coverings $\{U_i\to U\}_{i\in I}$ are collections of morphisms such that the induced morphism
        \[\bigsqcup_{i\in I} U_i\to U\]
    is crisp.
\end{cdef}

In this setting, we will compare the crisp topology to the fpqc topology, highlighting two aspects:
Firstly, we will study properties of morphisms of schemes which can be checked locally on the target in the crisp topology. This breaks down to examining which properties of morphisms of schemes descend along crisp quasi-compact morphisms of schemes in the Zariski topology. We present a theorem encompassing answers to this question.

\begin{cthm}{VII}[Theorem~\ref{crispt-descent} and Corollary~\ref{crisp-loc-props}]
     Let $f\colon X\to Y$ be a morphism of $S$-schemes and let $\{Y_i\to Y\}_{i\in I}$ be a crisp covering of $Y$. If for all $i$, the projection $Y_i\times_Y X\to Y_i$ has property \textbf{P}, then $f$ has property \textbf{P} as well, where \textbf{P} is one of the following properties:
    \begin{assertionlist}
        \item surjective
        \item injective
        \item bijective
        \item a morphism with set-theoretically finite fibres
        \item open
        \item closed
        \item a homeomorphism
        \item quasi-compact
        \item quasi-separated
        \item separated
        \item (locally) of finite type
        \item (locally) of finite presentation
        \item smooth
        \item unramified
        \item étale
        \item an isomorphism
        \item a monomorphism
        \item an open immersion
        \item proper
        \item quasi-finite
    \end{assertionlist}
\end{cthm}

Secondly, we will show that the crisp topology is subcanonical, drawing back to a similar statement made by André and Fiorot for affine schemes (cf. \cite[Theorem 3.2]{AF21}).

This paper is divided into three main chapters: In Chapter 1, we lay the foundation for any scheme-theoretic considerations by studying crisp ring homomorphisms and module maps. In 1970, Jean-Pierre Olivier gave a list of equivalent characterisations of crisp ring homomorphisms, albeit without proof (cf. \cite[Proposition 2.2]{Oli70}). Some of the proofs were already provided by Yves André and Luisa Fiorot in 2021 (cf. \cite{AF21}), and in this paper, we add a few more, which are partly based on a correspondence with André and Fiorot. Over time, there have appeared even more characterisations, notably in publications by Melvin Hochster and Joel Roberts (cf. \cite{HR74}), Yves André and Luisa Fiorot (cf. \cite{AF21}) as well as by Rüdiger Göbel and Jan Trlifaj (cf. \cite{GT06}). All of these results will be brought together in Theorem \ref{crisp-equivalences}, where we have added more equivalences of our own. The second chapter is all about crisp morphisms of schemes: Here, we will prove the above theorems about basic properties of such morphisms. The third chapter is dedicated to the crisp Grothendieck topology.\\

\noindent{\scshape Acknowledgements.\ }

I would like to express gratitute for my advisor Torsten Wedhorn's support while writing this paper. Moreover, I want to thank Chirantan Chowdhury, Yanik Kleibrink, Michelle Klemt, Catrin Mair, Christopher Lang and Alexander von Papen for proofreading different versions of the paper. Last but not least, I want to thank Yves André and Luisa Fiorot for helpful conversations via email.

    \tableofcontents
    \section{On crisp morphisms of rings and modules}\label{sec:two}

Let us briefly recall the definitions of crisp ring homomorphisms and module maps.

\begin{Definition}\label{crisp-module-map}
    Let $A$ be a ring and $M$ and $N$ be $A$-modules. An $A$-linear map $u\colon M\to N$ is called \textit{crisp} if the morphism 
        \[u\otimes \id_P \colon M\otimes_A P\to N\otimes_A P\]
    is injective for all $A$-modules $P$, cf. \cite[p. 86]{Laz69}.
\end{Definition}

\begin{Definition}\label{crisp-ring-hom}
    Let $A$ and $B$ be rings. A ring homomorphism $\varphi\colon A\to B$ is called \textit{crisp} if it is crisp as a morphism in ($A$-Mod), cf. \cite[Definition 1.1]{Oli70}.
\end{Definition}

\begin{Lemma}\label{ff-crisp-rings}
Any faithfully flat ring homomorphism is crisp.
\end{Lemma}

\begin{proof}
    Cf. \cite[\href{https://stacks.math.columbia.edu/tag/05CK}{Tag 05CK}]{stacks-project}.
\end{proof}

\begin{Lemma}\label{crisp-ring-hom-surj}
    Let $A$ and $B$ be rings. If the ring homomorphism $\varphi\colon A\to B$ is crisp, then the induced morphism of affine schemes $^a\varphi\colon \Spec B\to \Spec A$ is surjective.
\end{Lemma}

\begin{proof}
    For any $y\in\Spec A$, the ring homomorphism $\kappa(y)\to\kappa(y)\otimes_A B$ is injective by assumption. The affine scheme $\Spec \kappa(y)\otimes_A B$ has the underlying topological space $\{x\in\Spec B\colon \,^a\varphi(x)=y\}$. Assume $^a\varphi$ is not surjective. Then for some $y$, $$|\Spec \kappa(y)\otimes_A B|=\emptyset=|\Spec \{*\}|,$$ hence $\kappa(y)\otimes_A B=\{*\}$, but this contradicts injectivity of $\kappa(y)\to\kappa(y)\otimes_A B$.
\end{proof}

\subsection{Permanence properties}
We start with some permanence properties of crisp module maps as stated but not proven in \cite[Proposition 1.2]{Oli70}. Initially, Olivier stated these permanence properties only for ring homomorphisms in the source above. We deem it worthwhile to state them in the more general case of module maps, as this naturally extends to the case of ring homomorphisms (where the Properties (3) and (4) coincide).

\begin{Proposition}\label{permanence-crisp-rings}
    Crispness as a property of morphisms of modules over a ring $A$ has the following permanence properties:
    \begin{assertionlist}
        \item Being crisp is stable under composition.
        \item If $M\xrightarrow{u} M'$ and $M\xrightarrow{v} M''$ are maps of $A$-modules such that $v\circ u$ is crisp, then $u$ is crisp.
        \item Being crisp is stable under taking tensor products. 
         \item Being crisp is stable under base change.
         \item Being crisp is stable under crisp descent, i.e. if $A\xrightarrow{\varphi}B$ is a crisp ring homomorphism, then a map of $A$-modules $M\to M'$ is crisp if (and only if) the map  $M\otimes_A B\to M'\otimes_A B$ is crisp.
        \item Let $\mathcal{I}$ be a small filtered category and let $M, N\colon \mathcal{I}\to (A\textnormal{-Mod})$ be diagrams of $A$-modules written as $i\mapsto M_i$, resp. $i\mapsto N_i$. Consider a morphism of functors $u\colon M\to N$. If $M_i\to N_i$ is a crisp $A$-linear map for any $i\in I$, then
            \[\underset{i\in\mathcal{I}}{\colim}\,M_i\xrightarrow{\underset{i\in\mathcal{I}}{\colim}(u)} \underset{i\in\mathcal{I}}{\colim}\,N_i\]
        is a crisp $A$-linear map. If $(u_i\colon M_i\to N_i)_{i\in I}$ is a family of morphisms of $A$-modules such that $u_i$ is crisp for any $i\in I$, then the induced map of $A$-modules
            \[\bigoplus_{i\in I} u_i\colon \bigoplus_{i\in I} M_i\to \bigoplus_{i\in I} N_i\]
        is crisp.
    \end{assertionlist}
\end{Proposition}

\begin{proof}
     (1) + (2) follow directly from the fact that compositions of injective maps are injective again and that when the composition of two maps is injective, so is the first map.
    
    (3)\quad If $M\to M'$ is a crisp map of $A$-modules, then we obtain for any $A$-modules $N$ and $P$ that                 \begin{equation*}M\otimes_A(N\otimes_A P) = (M\otimes_A N)\otimes_A P \to (M'\otimes_A N)\otimes_A P = M'\otimes_A(N\otimes_A P)\end{equation*} 
    is injective as $N\otimes_A P$ is an $A$-module. Thus, $M\otimes_A N \to M'\otimes_A N$ is crisp.    

    (4)\quad If $u\colon M\to M'$ is a crisp map of $A$-modules, $B$ a ring and $\varphi\colon A\to B$ a ring homomorphism, then we obtain for any $B$-module $P$ that the $B$-linear map
        \begin{equation*}(M\otimes_A B)\otimes_B P = M\otimes_A P\to M'\otimes_A P = (M'\otimes_A B)\otimes_B P\end{equation*} 
    is injective by crispness of $u$. Thus, $M\otimes_A B \to M'\otimes_A B$ is a crisp $B$-linear map.
    
    (5)\quad If $M\to M'$ is a crisp map of $A$-modules, then $M\otimes_A B\to M'\otimes_A B$ is a crisp $B$-linear map by (4). Now, let $M\otimes_A B\to M'\otimes_A B$ be crisp as a $B$-linear map and consider the following commutative diagram of $A$-modules
        \begin{equation*}
            \begin{tikzcd}
                M\arrow[rightarrow]{r}\arrow[rightarrow]{d}&M'\arrow[rightarrow]{d}\\
                M\otimes_A B\arrow[rightarrow]{r}& M'\otimes_A B.
            \end{tikzcd}
        \end{equation*}
    The vertical arrows are crisp as $A\to B$ is, and the lower horizontal map is crisp as well: For any $A$-module $P$, the map of $B$-modules
        \[(M\otimes_A B)\otimes_A P = (M\otimes_A B)\otimes_B (P\otimes_A B)\to (M'\otimes_A B)\otimes_B (P\otimes_A B) = (M'\otimes_A B)\otimes_A P\]
    is injective by crispness of $M\otimes_A B\to M'\otimes_A B$ and also injective as an $A$-linear map. Thus, the composition $M\to M\otimes_A B\to M'\otimes_A B$ is crisp by (1). Because the diagram commutes, we obtain that the composition $M\to M'\to M'\otimes_A B$ is crisp, with which we get that $M\to M'$ is crisp by (2).
    
    (6) is due to the fact that tensor products commute with colimits.\end{proof}

For ring homomorphisms, we furthermore have that crispness satisfies being local in the following sense:

\begin{Proposition}\label{crisp-local-rings}
    Let $A$ and $B$ be rings. A ring homomorphism $\varphi \colon A\to B$ is crisp if and only if the induced ring homomorphisms $A_\mathfrak{p}\to A_\mathfrak{p} \otimes_A B $ are crisp for all prime ideals $\mathfrak{p}\subset A$, cf. \cite[Proposition 1.2]{Oli70}.
\end{Proposition}

\begin{proof}
    ``$\Rightarrow$'': Let $\mathfrak{p}\subset A$ be a prime ideal and $M$ an $A_\mathfrak{p}$-module. For the ring homomorphism $A_\mathfrak{p}\to A_\mathfrak{p} \otimes_A B $ to be crisp, it has to hold that 
        \begin{equation*}
            M = M\otimes_{A_\mathfrak{p}} A_\mathfrak{p}\to M\otimes_{A_\mathfrak{p}} A_\mathfrak{p} \otimes_A B = M\otimes_A B
        \end{equation*} 
    is injective. This is true because $\varphi$ is crisp by assumption.

    ``$\Leftarrow$'': Let $M$ be an $A$-module. With the crispness of $A_\mathfrak{p}\to B_\mathfrak{p}$, we get that the map
        \begin{equation*}
            M_\mathfrak{p}\otimes_{A_\mathfrak{p}} A_\mathfrak{p}\longrightarrow M_\mathfrak{p}\otimes_{A_\mathfrak{p}} B_\mathfrak{p} 
        \end{equation*}
    is injective. As
        \begin{equation*}
            M_\mathfrak{p}\otimes_{A_\mathfrak{p}} A_\mathfrak{p}\cong M\otimes_A A_\mathfrak{p}\quad\text{and}\quad
            M_\mathfrak{p}\otimes_{A_\mathfrak{p}} B_\mathfrak{p}\cong (M\otimes_A B)\otimes_A A_\mathfrak{p},
        \end{equation*}
    the map 
        \[M\otimes_A A_\mathfrak{p}\longrightarrow (M\otimes_A B)\otimes_A A_\mathfrak{p}\]
    is injective as well. This holds in particular for any maximal ideal of $A$, such that $M\to M\otimes_A B$ is injective and thus, $A\to B$ is crisp.
\end{proof}

\subsection{Equivalent characterisations of crisp ring homomorphisms}

The following lemma will be necessary for the upcoming main theorem of this subsection: 

\begin{Lemma}\label{inj-surj-iso-crisp}
    Let $\varphi\colon A\to B$ be a crisp ring homomorphism and $u\colon M\rightarrow M'$ be a map of $A$-modules such that the $B$-linear map $u\otimes \id_B$ is 
    \begin{assertionlist}
        \item injective. Then $u$ is injective.
        \item surjective. Then $u$ is surjective.
        \item an isomorphism of $B$-modules. Then $u$ is an isomorphism of $A$-modules.
    \end{assertionlist}
\end{Lemma}

\begin{proof}
    (1)\quad For the injectivity of $u$, we observe that $M\hookrightarrow M\otimes_A B$ and $M'\hookrightarrow M'\otimes_A B$ are injective $A$-linear maps as $\varphi$ is crisp and that $M\otimes_A B\hookrightarrow M'\otimes_A B$ is injective as a $B$-linear (and as an $A$-linear) map by the premise of this proposition. In the commutative diagram of $A$-modules
        \[ \begin{tikzcd}
            M \arrow[hookrightarrow]{r} \arrow[swap]{d}{u} & M\otimes_A B \arrow[hookrightarrow]{d} \\%
            M' \arrow[hookrightarrow]{r}& M'\otimes_A B,
        \end{tikzcd}\]
    a diagram chase yields that $u$ has to be injective.

    (2)\quad As $u\otimes \id_B$ is surjective, the sequence of $B$-modules
        \[\Ker(u\otimes \id_B)\longrightarrow M\otimes_A B\longrightarrow M'\otimes_A B\longrightarrow 0 \]
    is exact, implying that $\Coker(u)\otimes_A B=\Coker(u\otimes \id_B)=0$. As $\varphi$ is crisp, we obtain that $\Coker(u)=0$ with which $u$ is surjective.

    (3)\quad follows immediately from (1) and (2).
\end{proof}

We now come to the main theorem of this section:

\begin{Theorem}\label{crisp-equivalences}
    For a ring homomorphism $\varphi\colon A\to B$ and its induced base change functor $\varphi^*\colon (A\textnormal{-Mod})\to (B\textnormal{-Mod})$ the following are equivalent:
    \begin{equivlist}
        \item The map $\varphi$ is crisp, i.e. for any $A$-module $M$, the induced map $M\to M\otimes_A B$ is injective.
        \item For any $A$-algebra $R$, the induced map $R\to R\otimes_A B$ is injective.
        \item For any $A$-module of finite presentation $M$, the induced map $M\to M\otimes_A B$ is injective.
        \item The functor $\varphi^*$ is faithful.
        \item For every map of $A$-modules $u\colon M\to M'$, $u$ is injective whenever $u\otimes\id_B$ is an injective $B$-linear map.
        \item For every map of $A$-modules $u\colon M\to M'$, $u$ is crisp whenever $u\otimes\id_B$ is a crisp $B$-linear map.
        \item For every map of $A$-modules $u\colon M\to M'$, $u$ is an isomorphism whenever $u\otimes\id_B$ is an isomorphism of $B$-modules.
        \item For every sequence $\mathcal{S}$ of $A$-modules
            \[0\longrightarrow M'\longrightarrow M\longrightarrow M''\longrightarrow 0,\]
        $\mathcal{S}$ is exact whenever $\varphi^*(\mathcal{S})$ is an exact sequence of $B$-modules.
        \item For any $A$-module $M$, the sequence
            \[M\longrightarrow M\otimes_A B\rightrightarrows M\otimes_A B\otimes_A B\]
        is exact.
        \item For any $A$-algebra $R$, the sequence
            \[R\longrightarrow R\otimes_A B\rightrightarrows R\otimes_A B\otimes_A B\]
        is exact.
        \item A family of $A$-modules and $A$-linear maps $(M^i,d^i\colon M^i\to M^{i+1})_{i\in \Z}$ is a complex whenever $(\varphi^*(M^i),\varphi^*(d^i))_{i\in \Z}$ is a complex of $B$-modules.
        \item For every complex $M^{\bullet}$ of $A$-modules, the induced map in cohomology
            \[H^i(M^{\bullet})\to H^i(M^{\bullet}\otimes_A B)\]
        is injective in every degree $i$.
        \item The map $\varphi$ is injective and for any $B$-module $F$ of finite presentation, the functor $\Hom(F,-)$ preserves exactness of the short exact sequence of $A$-modules
            \[0\longrightarrow A\longrightarrow B\longrightarrow \faktor{B}{A}\longrightarrow 0.\]
        \item The map $\varphi$ is injective and for all $m,n\in\N$ and all systems of $A$-linear equations $\mathcal{E}$ in the variables $(x_j)_{1\leq j\leq n}$ with $a_i\in A$ for $1\leq i\leq m$ as well as $c_{ij}\in A$ for $1\leq i\leq m, 1\leq j\leq n$
            \[\sum_{j=1}^n c_{ij}x_j = a_i\quad (1\leq i\leq m),\]
        it holds that $\mathcal{E}$ has a solution in $A$ whenever $\mathcal{E}$ has a solution in $B$.
    \end{equivlist}
\end{Theorem}

\begin{proof}[Proof of Theorem \ref{crisp-equivalences}]
    $\textnormal{(i)}\Leftrightarrow \textnormal{(ii)}$\quad Cf. \cite[Proposition 2.2]{AF21}.

    $\textnormal{(i)}\Leftrightarrow \textnormal{(iii)}$\quad Follows from the facts that any $A$-module can be expressed as a filtered colimit of some modules of finite presentation and that colimits both commute with tensor products and are exact.

    $\textnormal{(i)}\Leftrightarrow \textnormal{(iv)}$\quad Cf. \cite[Proposition 2.2]{AF21}.

    $\textnormal{(i)}\Rightarrow \textnormal{(v)}$\quad Assume that $\varphi$ is crisp and let $u\colon M\to M'$ be a map of $A$-modules. Moreover, assume that $u\otimes\id_B$ is an injective $B$-linear map. Then, the injectivity of $u$ follows immediately with Lemma \ref{inj-surj-iso-crisp}.

    $\textnormal{(v)}\Rightarrow \textnormal{(i)}$\footnote{Cf. \cite{Fio24}.}\quad Assume that any $A$-linear map $u$ is injective whenever its base change $\varphi^*(u)=u\otimes \mathrm{id}_B$ is and consider the base change diagram for an arbitrary $A$-module $M$:
        \[\begin{tikzcd}
            M\arrow[r]\arrow[d] & M\otimes_A B\arrow[d]\\
            M\otimes_A B \arrow[r] & M\otimes_A (B\otimes_A B).
        \end{tikzcd}\]
    We note that $B\to B\otimes_A B$ is a split injective $A$-linear map since 
        \[(b\otimes b'\mapsto b\cdot b')\circ (b\mapsto b\otimes 1) = \id_B.\]
    As tensoring preserves injectivity for split injective maps, we get that 
        \[M\otimes_A B \longrightarrow M\otimes_A (B\otimes_A B)\]
    is injective with which $M\to M\otimes_A B$ is injective by assumption.

    $\textnormal{(i)}\Rightarrow \textnormal{(vi)}$\quad Let $\varphi$ be crisp and let $u\colon M\to M'$ be an $A$-linear map for which $u\otimes\id_B$ is a crisp $A$-linear map. Then, $u$ is crisp by Proposition \ref{permanence-crisp-rings} (5).

    $\textnormal{(vi)}\Rightarrow \textnormal{(i)}$\footnotemark[\value{footnote}]\quad We have already seen that $B\to B\otimes_A B$ is a split injective $A$-linear map. Thus, it is crisp, and so $A\to B$ is crisp.
    
     $\textnormal{(i)}\Rightarrow \textnormal{(vii)}$\quad This immediately follows from Lemma \ref{inj-surj-iso-crisp}.
    
    $\textnormal{(vii)}\Rightarrow \textnormal{(iv)}$\footnotemark[\value{footnote}]\quad Choose some morphism of $A$-modules $u\colon M\to N$ and assume that $u\otimes\id_B = 0$ as a $B$-linear map. For $\varphi^*$ to be faithful, it suffices to show that $u = 0$. For this, consider the exact sequence of $A$-modules
        \[M\overset{u}{\longrightarrow} N\overset{\pi}{\longrightarrow}\Coker (u)\longrightarrow 0.\]
    Tensoring with $B$ yields the exact sequence of $B$-modules
        \[M\otimes_A B\xrightarrow{u\otimes \id_B = 0} N\otimes_A B\xrightarrow{\pi\otimes \id_B}\Coker (u)\otimes_A B\longrightarrow 0.\]
    Then, $\pi\otimes \id_B$ needs to be an isomorphism of $B$-modules, so $\pi$ is an isomorphism of $A$-modules. Then, $\Im(u) = 0$ and thus $u = 0$.
    \footnotetext{Cf. \cite{Fio24}.}
    
    $\textnormal{(i)}\Rightarrow \textnormal{(viii)}$\quad Let $\varphi$ be crisp and consider a sequence $\mathcal{S}$ of $A$-modules
        \[0\longrightarrow M'\overset{u}{\longrightarrow} M\overset{v}{\longrightarrow} M''\longrightarrow 0\]
    such that the induced sequence $\varphi^*(\mathcal{S})$ of $B$-modules given by
        \[0\longrightarrow M'\otimes_A B\xrightarrow{\varphi^*(u)} M\otimes_A B \xrightarrow{\varphi^*(v)} M''\otimes_A B\longrightarrow 0\]
    is exact. As Lemma \ref{inj-surj-iso-crisp} yields that then, $u$ is injective and $v$ is surjective, it only remains to show that $\Im(u) = \Ker(v)$.
    We first note that 
        \[\Im(u)\otimes_A B = \Im(\varphi^*(u))\]
    as both terms are isomorphic as $B$-modules to $M'\otimes_A B$, which follows from the injectivity of both $u$ and $\varphi^*(u)$.\\
    Now, choose some $m\in M$ with $v(m) = 0$. Then, 
        \[0 = v(m)\otimes 1 = \varphi^*(v)(m\otimes 1)\]
    in $M''\otimes_A B$, such that $m\otimes 1$ is an element of 
        \[\Ker(\varphi^*(v)) = \Im(\varphi^*(u)) = \Im(u)\otimes_A B.\] Thus, $m$ lies in the image of $u$. \\
    Conversely, for any $m\in M$ with $m = u(m')$ for some $m'\in M'$, 
        \[m\otimes 1 = u(m')\otimes 1 = \varphi^*(u)(m'\otimes 1)\] 
    is an element of $\Im(\varphi^*(u)) = \Ker(\varphi^*(v))$. Because of the injectivity of $M\to M\otimes_A B$, $m$ now also needs to lie in $\Ker(v)$.

    $\textnormal{(viii)}\Rightarrow \textnormal{(i)}$\quad We assume that for every sequence of $A$-modules $\mathcal{S}$ as above, $\mathcal{S}$ is exact whenever $\varphi^*(\mathcal{S})$ is. Let $u\colon M\to M'$ be an $A$-linear map for which $u\otimes\id_B$ is an injective $B$-linear map. Then, the sequence of $B$-modules
        \[\begin{tikzcd}
            0\arrow[r] & M\otimes_A B \arrow[r] & M'\otimes_A B\arrow[r]&\Coker(M\otimes_A B\to M'\otimes_A B) \arrow[d, equal] & \\
            & & &\Coker(M\to M')\otimes_A B\arrow[r] & 0
        \end{tikzcd}\]
    is exact, implying that the sequence of $A$-modules
        \[0\longrightarrow M\longrightarrow M'\longrightarrow \Coker(M\to M')\longrightarrow 0\]
    is exact by assumption. Hence, $u$ is injective, which yields together with $\textnormal{(i)}\Leftrightarrow \textnormal{(iv)}$ that $f$ is crisp.

    $\textnormal{(i)}\Leftrightarrow \textnormal{(ix)}$\quad Cf. \cite[Proposition 2.2]{AF21}.

    $\textnormal{(i)}\Leftrightarrow \textnormal{(x)}$\quad Cf. \cite[Proposition 2.2]{AF21}.

    $\textnormal{(i)}\Leftrightarrow \textnormal{(xi)}$\quad Let $\varphi$ be crisp and let $(M^i,d^i:M^i\to M^{i+1})_{i\in \Z}$ be a family of $A$-modules and $A$-linear maps. It is clear that $(\varphi^*(M^i),\varphi^*(d^i))$ is a complex whenever $(M^i,d^i)$ is. Assume conversely that $(\varphi^*(M^i),\varphi^*(d^i))$ is a complex and consider the commutative diagram
        \[\begin{tikzcd}[column sep = 3.8em, row sep = 2em]
            \cdots \arrow[r]& M^i\arrow[r, "d^i"] \arrow[d, hook] & M^{i+1}\arrow[r, "d^{i+1}"]\arrow[d, hook] & M^{i+2}\arrow[r] \arrow[d,hook]& \cdots \\
            \cdots \arrow[r]& M^i\otimes_A B\arrow[r, "\varphi^*(d^i)"] & M^{i+1}\otimes_A B\arrow[r, "\varphi^*(d^{i+1})"] & M^{i+2}\otimes_A B\arrow[r] & \cdots
        \end{tikzcd}\]
    where the injectivity of the vertical arrows is due to the crispness of $\varphi$. Choose some $m\in M^i$. As
        \[0 = (\varphi^*(d^{i+1})\circ \varphi^*(d^i))(m\otimes 1) = (d^{i+1}\circ d^i)(m)\otimes 1\] 
    by assumption, the injectivity of the rightmost vertical arrow above yields $(d^{i+1}\circ d^i)(m) = 0$ such that $(M^i,d^i:M^i\to M^{i+1})$ is a complex as wished.

    Now, assume that $\varphi$ is not crisp, i.e. that there is some $A$-module $M$ such that $M\to M\otimes_A B$ is not injective. We choose some non-trivial $m\in M$ such that $m\otimes 1 = 0$. Let us consider the sequence of $A$-modules
        \[\dots\longrightarrow 0\longrightarrow 0\longrightarrow \langle m\rangle \hookrightarrow M\xrightarrow{\id_M} M\longrightarrow 0\longrightarrow 0\dots,\]
    which is not a complex of $A$-modules. Tensoring this sequence with $B$ yields the sequence
        \[\dots\longrightarrow 0\longrightarrow 0\longrightarrow \langle m\rangle\otimes_A B  \overset{0}{\hookrightarrow} M\otimes_A B\xrightarrow{\id_M\otimes\id_B} M\otimes_A B\longrightarrow 0\longrightarrow 0\dots,\]
    which is a complex of $A$-modules, as can be checked.
    
    $\textnormal{(i)}\Rightarrow \textnormal{(xii)}$\quad Cf. \cite[Corollary 6.6]{HR74}.
    
    $\textnormal{(xii)}\Rightarrow \textnormal{(vii)}$\quad Let $u\colon M\to M'$ be a map of $A$-modules such that $u\otimes\id_B$ is an isomorphism of $B$-modules. We obtain complices of $A$-modules
        \[\dots\longrightarrow 0\longrightarrow M\overset{u}{\longrightarrow} M'\longrightarrow 0\longrightarrow\dots\quad\textnormal{and}\]
        \[\dots\longrightarrow 0\longrightarrow M\otimes_A B\xrightarrow{u\otimes\id_B} M'\otimes_A B\longrightarrow 0\longrightarrow\dots\]
    where the latter complex has trivial cohomology everywhere by virtue of the isomorphism $u\otimes\id_B$ of $B$-modules, which is also an isomorphism of $A$-modules. Our assumption now yields that the former complex also has trivial cohomology everywhere, making $u$ into an isomorphism as wished.

    $\textnormal{(i)}\Leftrightarrow \textnormal{(xiii)} \Leftrightarrow \textnormal{(xiv)}$\quad  Cf. \cite[Definition 1.2.4 and Lemma 1.2.13]{GT06}.
\end{proof}

\subsection{Some examples}
The following examples show in particular that there are crisp ring homomorphisms which are not faithfully flat or not split injective.
\begin{Example}\label{crisp-examples}
    \begin{assertionlist}
        \item Let $A$ be a ring. We consider the inclusion $\iota\colon A\hookrightarrow A[[X]]$. Clearly, this map is a split injection as the composition
            \[\begin{tikzcd}[column sep = 6em]
                A\arrow[r, hook, "a\mapsto a"]& A[[X]]\arrow[r,"\sum_{\N}a_iX^i\,\mapsto\, a_0"] & A
            \end{tikzcd}\]
        is the identity on $A$. It is easy to see that this implies crispness of $\iota$ as tensoring is exact for split exact sequences. Of course, this construction also works to show crispness of $A\to A[(X_i)_{i\in I}]$ or of $A\to A[[(X_i)_{i\in I}]]$.
        \item Let $A$ be a ring and consider the ring homomorphism
            \[A\xrightarrow{\varphi} \prod_{\mathfrak{p}\in\Spec A}A_{\mathfrak{p}}.\]
        We claim that $\varphi$ is crisp. Using Proposition \ref{crisp-local-rings}, this map is crisp if and only if for any prime ideal $\mathfrak{q}$ in $A$, the induced map 
            \[A_{\mathfrak{q}}\xrightarrow{\varphi_{\mathfrak{q}}}\left(\prod_{\mathfrak{p}\in\Spec A}A_{\mathfrak{p}}\right)_{\mathfrak{q}}\]
        is crisp. We note that the composition
            \[A_{\mathfrak{q}}\xrightarrow{\varphi_{\mathfrak{q}}}\left(\prod_{\mathfrak{p}\in\Spec A}A_{\mathfrak{p}}\right)_{\mathfrak{q}}\xrightarrow{\mathrm{proj}_{\mathfrak{q}}}\left(A_{\mathfrak{q}}\right)_{\mathfrak{q}} = A_{\mathfrak{q}}\]
        is an isomorphism and thus crisp. Proposition \ref{permanence-crisp-rings} (2) now yields that 
            \[A_{\mathfrak{q}}\xrightarrow{\varphi_{\mathfrak{q}}}\left(\prod_{\mathfrak{p}\in\Spec A}A_{\mathfrak{p}}\right)_{\mathfrak{q}}\]
        is crisp which means that $\varphi$ is as well. 
        \item Let $A$ be a ring and $M$ an $A$-module that is not flat (e.g. $A/I$ for an ideal
        $I\subset A$ where $I^2\neq I$). We set $B\coloneqq A\oplus M$. Then, $B$ is an $A$-module, and we can endow it with a
        ring structure as usual: The multiplication map 
            \begin{equation*}
                B\times B\to B,\,\,(a,m)\cdot (a',m')\coloneqq (a\cdot a', a\cdot m'+a'\cdot m) 
                \end{equation*}
        turns $B$ into a ring with multiplicative identity $(1,0)$. Furthermore, the map
            \begin{equation*}
                \varphi\colon A\to B,\quad a\mapsto (a,0)
            \end{equation*}
        turns $B$ into an $A$-algebra. As the projection $\psi\colon B=A\oplus M\to A$ is a ring homomorphism for which it holds that $\psi\circ \varphi = \id_A$, $\varphi$ is a split injection and thus crisp.  However, $\varphi$ is not (faithfully) flat: As $M$ is not a flat $A$-module, there is some exact sequence of $A$-modules
            \begin{equation*}
                0\longrightarrow N'\overset{u}{\longrightarrow} N\longrightarrow N''\longrightarrow 0
            \end{equation*}
        for which 
            \begin{equation*}
                0\longrightarrow N'\otimes_A M\xrightarrow{u\otimes\id_M} N\otimes_A M\longrightarrow N''\otimes_A M\longrightarrow 0
            \end{equation*}
        is not exact, i.e. for which $u\otimes\id_M$ is not injective. Now, assume that $\varphi$ is flat, i.e. that $A\oplus M$ is a flat $A$-algebra. Then, it is a flat $A$-module such that this time, the sequence
            \begin{equation*}
                0\longrightarrow N'\otimes_A (A\oplus M)\xrightarrow{u\otimes(\id_A,\id_M)} N\otimes_A (A\oplus M)\longrightarrow N''\otimes_A (A\oplus M)\longrightarrow 0
            \end{equation*}
        is indeed exact which means that the map $u\otimes(\id_A,\id_M)$ is injective. Since tensor products commute with direct sums, we obtain that the map $(u\otimes\id_A,u\otimes\id_M)$ is injective. This holds only if $u\otimes\id_M$ is injective, which contradicts the fact that $M$ is not flat. Hence, $A\oplus M$ is not a flat $A$-module, implying that $\varphi$ is not faithfully flat but certainly crisp.\footnote{I would like to thank StackExchange user \href{https://math.stackexchange.com/users/245104/mohan?tab=profile}{Mohan} for this \href{https://math.stackexchange.com/questions/4714001/example-of-a-universally-injective-ring-morphism-which-is-not-faithfully-flat/4714233\#4714233}{example}.}
        \item Let us consider the exact sequence of $\Z$-modules 
            \begin{equation*}
                0\longrightarrow \bigoplus_{n\in\N}\Z \overset{u}{\longrightarrow} \prod_{n\in\N}\Z\overset{v}{\longrightarrow} \Coker(u)\longrightarrow 0.
            \end{equation*}
        It is easy to show that all of the modules above are flat $\Z$-modules as $\Z$ is a principal ideal domain, and they are torsion-free. Then, \cite[\href{https://stacks.math.columbia.edu/tag/058M}{Tag 058M}]{stacks-project} yields that the sequence is universally exact, i.e. that the map $u$ is crisp. We show that $u$ cannot be a split injection: Assuming that $u$ is indeed split injective, we obtain that the sequence above is split exact. Then, $v$ must be split surjective. Now, consider 
            \[x\coloneqq v(2,2^2,2^3,\dots)\neq 0\in \Coker(u).\] 
        Because it holds that
            \[x = v(0,0,\dots,0,2^{n},2^{n+1},2^{n+2},\dots)\] 
        in $\Coker(u)$ for any $n\in\N$, we obtain that the image of $x$ under any map $$w\colon \Coker(u)\to \prod_{n\in\N}\Z$$needs to be zero as $2^n\,|\,w(x)$ for any $n\in\N$, contradicting the fact that $v$ is split surjective. Thus, $u$ is a crisp module map that is not split injective, cf. \cite[\href{https://stacks.math.columbia.edu/tag/058N}{Tag 058N}]{stacks-project}.
    \end{assertionlist}
\end{Example}

\subsection{Crisp descent for modules and algebras}

This subsection studies the properties of rings, ring algebras and modules which descend along crisp ring homomorphisms. This problem has been investigated before, e.g. by Bachuki Mesablishvili (\cite{Mes00}, \cite{Mes02}), Gerhard Angermüller (\cite{Ang20}) and the Stacks Project authors (\cite[\href{https://stacks.math.columbia.edu/tag/08WE}{Tag 08WE}]{stacks-project}). We will compile their respective propositions and add two of our own: that being an integral or a smooth ring algebra descends along crisp ring homomorphisms. 

\begin{Proposition}\label{descent-module}
    Let $M$ be an $A$-module and $\varphi\colon A\to B$ be crisp. 
    \begin{assertionlist}
        \item If $M\otimes_A B$ is a finitely generated (resp. finitely presented) $B$-module, then $M$ is a finitely generated (resp. finitely presented) $A$-module.
        \item If $M\otimes_A B$ is (faithfully) flat as a $B$-module, then $M$ is (faithfully) flat as an $A$-module.
        \item If $M\otimes_A B$ is a vector bundle as a $B$-module, then $M$ is a vector bundle as an $A$-module.
        \item If $M\otimes_A B$ is a projective (not necessarily finitely generated) $B$-module, then $M$ is a projective $A$-module.
    \end{assertionlist}
\end{Proposition}

\begin{proof}
    Cf. \cite[\href{https://stacks.math.columbia.edu/tag/08XD}{Tag 08XD}]{stacks-project}, \cite[Theorem 3.2]{Mes02} and \cite{Ang20}.
\end{proof}

\begin{Proposition}\label{descent-algebras}
    Let $\varphi\colon A\to B$ be a crisp ring homomorphism and let $R$ be an $A$-algebra via any ring homomorphism $A\to R$. If $R\otimes_A B$ is a finite $B$-algebra (resp. a $B$-algebra of finite type, resp. of finite presentation) and $\varphi\colon A\to B$ is crisp, then $R$ is a finite $A$-algebra (resp. an $A$-algebra of finite type, resp. of finite presentation).
\end{Proposition}

\begin{proof}
    Cf.  \cite[\href{https://stacks.math.columbia.edu/tag/08XE}{Tag 08XE}]{stacks-project}.
\end{proof}

\begin{Proposition}\label{descent-integral}
    If $R\otimes_A B$ is integral over $B$ and $\varphi\colon A\to B$ is crisp, then $R$ is integral over $A$.
\end{Proposition}

\begin{proof}
   This proposition uses later, independent results for crisp morphisms of affine schemes: A ring homomorphism $\varphi$ is integral if and only if the corresponding morphism of affine schemes $^a\varphi$ is universally closed (cf. \cite[\href{https://stacks.math.columbia.edu/tag/01WM}{Tag 01WM}]{stacks-project}). This property descends along crisp morphisms of schemes (Theorem \ref{crispt-descent}), and a morphism of affine schemes is crisp if and only if the underlying ring homomorphism is crisp (Theorem \ref{pureqcqs-crisp}).
\end{proof}

\begin{Proposition}\label{descent-etale-etc}
    Let $\varphi\colon A\to B$ be a crisp ring homomorphism and let $R$ be an $A$-algebra via any ring homomorphism $A\to R$. If $R\otimes_A B$ is a (formally) unramified (resp. étale, resp. smooth) $B$-algebra, then $R$ is a (formally) unramified (resp. étale, resp. smooth) $A$-algebra.  
\end{Proposition}

\begin{proof}
    The proofs for the properties ``(formally) unramified'' and ``étale'' can be found in  \cite[\href{https://stacks.math.columbia.edu/tag/08XE}{Tag 08XE}]{stacks-project}.
    
    For smoothness, we show that if $\Spec R\otimes_A B$ is a smooth affine scheme over $\Spec B$, then $\Spec R$ is a smooth affine scheme over $\Spec A$. For the definition of crispness for affine schemes, we refer to Corollary \ref{equivalent-on-affines}. Let $\Spec R\otimes_A B$ be a smooth affine scheme over $\Spec B$. By \cite[\href{https://stacks.math.columbia.edu/tag/01V8}{Tag 01V8}]{stacks-project}, it suffices to show that 
    \begin{assertionlist}
      \item $\Spec R\to \Spec A$ is flat.
      \item $\Spec R\to \Spec A$ is locally of finite presentation.
       \item All scheme-theoretic fibres $(\Spec R)_a$ with $a\in \Spec A$ are smooth.
    \end{assertionlist}

    Statements (1) and (2) immediately follow from the respective statements for module maps, cf. Proposition \ref{descent-module}.

    For (3) we pursue the following strategy: We note that the fibres $(\Spec R\otimes_A B)_b$ for $b\in \Spec B$ are smooth over $\kappa(b)$. By surjectivity of $^a\varphi$ (cf. Lemma \ref{crisp-ring-hom-surj}), for any $a\in\Spec A$ there is some $b\in\Spec B$ with an induced map $\Spec \kappa(b)\to \Spec \kappa(a)$. This map is faithfully flat and since 
        \[(\Spec R\otimes_A B)_b=(\Spec R)_a\otimes_{\Spec \kappa(a)}\Spec \kappa(b),\]
    smoothness of the fibres $(\Spec R)_a$ now follows from faithfully flat descent of smoothness.
\end{proof}
    \newpage

\section{On crisp morphisms of schemes}\label{chap::three}

This chapter is dedicated to the new notion of crisp morphisms as in Definition \ref{crisp-morphism}. Here, we will prove Theorems \ref{ppintro} and \ref{crisp-context} as stated in the introduction. These theorems enable the reader to work with crisp morphisms of schemes as with other scheme morphisms.

Let us briefly recall the definitions of locally quasi-compact surjective morphisms and crisp morphisms of quasi-separated schemes. 

\begin{Lemma}\label{loc-qc-surjective}
    Let $f\colon X\to Y$ be a surjective morphism of schemes. Then, the following are equivalent (and hereafter referred to as being \textnormal{locally quasi-compact surjective} or \textnormal{lqcs}):
    \begin{equivlist}
        \itemsep0em 
        \item Every quasi-compact open subset of $Y$ is the image of a quasi-compact open subset of $X$.
        \item There is a covering $(V_i)_{i\in I}$ of $Y$ by open affine subschemes such that each $V_i$ is the image of a quasi-compact open subset of $X$. 
        \item For all $x\in X$, there is an open neighbourhood $U_x$ of $x$ in $X$ such that $f(U_x)$ is open in $Y$ and the restriction $U_x\to f(U_x)$ is quasi-compact.
        \item For all $x\in X$, there is a quasi-compact open neighbourhood $U$ of $x$ in $X$ such that $f(U)$ is open and affine in $Y.$
    \end{equivlist}
\end{Lemma}

\begin{proof}
    Cf. \cite[Proposition 2.33]{FA05}.
\end{proof}

\begin{Remark}
    Being locally quasi-compact surjective is stable under composition and local on the target (cf. \cite[Proposition 2.35]{FA05}) as well as stable under base change: 
    
    Let $f\colon X\to Y$ be an lqcs morphism of schemes and let $Y'\to Y$ be another scheme morphism. We choose an open affine covering $(V_i)_{i\in I}$ of $Y$ as in (ii) of Lemma \ref{loc-qc-surjective}. We obtain an open covering $(g^{-1}(V_i))_{i\in I}$ of $Y'$ which can be refined to an open affine covering $(V'_j)_{j\in J}$ of $Y'$.
    Any affine open $V'_j$ now maps into some affine open $V_i\subseteq Y$ and by assumption, there is some quasi-compact open $U_i\subseteq X$ with $f(U_i)=V_i$. Then, $U_i\times_{V_i} V_j'\subseteq X\times_Y Y'$ is a quasi-compact open with $f'(U_i\times_{V_i} V_j')=V_j'.$
\end{Remark}

\begin{Remark}
    Any quasi-compact and surjective morphism of schemes and any faithfully flat morphism locally of finite presentation is locally quasi-compact surjective.
\end{Remark}

\begin{Definition}\label{crisp-morphism}
    A surjective morphism of quasi-separated schemes $f\colon X\to Y$ is called \textit{crisp}, if it is locally quasi-compact surjective and satisfies the following equivalent properties:
    \begin{equivlist}
        \item for every open affine subscheme $V\subseteq Y$, there is a quasi-compact open $U\subseteq f^{-1}(V)$ with $f(U) = V$ such that there is a finite affine open covering $U_{1},\dots, U_{r}$ of $U$ for which the ring homomorphism corresponding to the morphism of affine schemes
            \[\bigsqcup_{i=1} ^r U_{i}\longrightarrow V\]
        is crisp.
        \item for every open affine subscheme $V\subseteq Y$, there is a quasi-compact open $U\subseteq f^{-1}(V)$ with $f(U) = V$ such that for every finite affine open covering $U_{1},\dots, U_{r}$ of $U$, the ring homomorphism corresponding to the morphism of affine schemes
            \[\bigsqcup_{i=1} ^r U_{i}\longrightarrow V\]
        is crisp.
        \item there is an open affine covering $(V_j)_{j\in J}$ of $Y$ such that for all $j$, there is a quasi-compact open $U_j\subseteq f^{-1}(V_j)$ with $f(U_j) = V_j$ such that there is a finite open affine covering $U_{j,1},\dots, U_{j,r_j}$ of $U_j$ for which the ring homomorphism corresponding to the morphism of affine schemes
            \[\bigsqcup_{i=1} ^{r_j} U_{j,i}\longrightarrow V_j\]
        is crisp.
        \item there is an open affine covering $(V_j)_{j\in J}$ of $Y$ such that for all $j$, there is a quasi-compact open $U_j\subseteq f^{-1}(V_j)$ with $f(U_j) = V_j$ such that for every finite open affine covering $U_{j,1},\dots, U_{j,r_j}$ of $U_j$, the ring homomorphism corresponding to the morphism of affine schemes
            \[\bigsqcup_{i=1} ^{r_j} U_{j,i}\longrightarrow V_j\]
        is crisp.
    \end{equivlist}
\end{Definition}

\begin{Remark}
    \begin{enumerate}[(1)]
        \item Note that the choice of the quasi-compact open subschemes mapping surjectively onto the affine open subschemes is made possible by $f$ being locally quasi-compact surjective. 
        \item If $f\colon X\to Y$ is crisp, this does \emph{not} imply that for any affine $U\subseteq X$ and $V\subseteq Y$ with $f(U) = V$, the ring homomorphism corresponding to $U\to V$ is crisp! 
    \end{enumerate}
\end{Remark}

\begin{Remark}[The garbage principle]\label{garbage}
    Since for any ring $A$ and $A$-algebras $B$ and $C$, the induced map $A\to B\times C$ is crisp if $A\to B$ is crisp, we can ``add any garbage'' to the image of a crisp ring homomorphism without destroying the crispness. Correspondingly, we obtain that whenever $f\colon X\to Y$ is a morphism of schemes such that $f|_U\colon U\to Y$ is crisp for some open subscheme $U\subseteq X$, then $f$ is crisp. We call this fact ``the garbage principle''.
\end{Remark}
 
\begin{Lemma}\label{crisp-product-rings}
    If $\varphi_1\colon A_1\to B_1$ and $\varphi_2\colon A_2\to B_2$ are crisp ring homomorphisms, then so is $\varphi_1\times\varphi_2$.
\end{Lemma}

\begin{proof}
    Geometrically, this is clear: It is shown in \cite[Theorem 5.12]{Mes04a} that a ring homomorphism is crisp in the sense of Definition \ref{crisp-ring-hom} if and only if the associated morphism of affine schemes is universally schematically dominant (i.e. pure), which is a property of scheme morphisms local on the target. Thus, passing to the open covering $(\Spec (A_1),\Spec(A_2))$ of $\Spec A_1\times A_2=\Spec A_1\sqcup \Spec A_2$, we obtain the statement.
\end{proof}

\begin{proof}[Proof of the equivalences in Definition \ref{crisp-morphism}]
    For the equivalences $\textnormal{(i)}\Leftrightarrow\textnormal{(ii)}$ and $\textnormal{(iii)}\Leftrightarrow\textnormal{(iv)}$, we note that it suffices to prove the following statement: Let $f\colon X\to Y$ be a morphism of schemes where $X$ is qcqs and $Y$ is affine. Then, there is a finite open affine covering $(V_i)_{i}$ of $X $ such that the ring homomorphism corresponding to 
        \[\bigsqcup_{i} V_i\to X\] 
    is crisp if and only if this is the case for every finite open affine covering $(U_j)_{j}$ of $X$.\\
    We only need to show one implication. Thus, let $(V_i)_{i}$ be a finite open affine covering of $X$ such that the ring homomorphism corresponding to
        \[\bigsqcup_{i} V_i\to X\]
    is crisp. Choose another finite open affine covering $(U_j)_{j}$ of $X$. We want to show that the ring homomorphism corresponding to 
        \[\bigsqcup_{j} U_j\to X\]
    is crisp. For this, we note the following: Since $X$ is quasi-separated, the intersections $U_j\cap V_i$ are quasi-compact. Hence, we find finitely many affine opens $(W^{(ij)}_k)_{k}$ covering $U_j\cap V_i$. Then, we have the equalities
        \[V_i = \bigcup_{j,k}W^{(ij)}_k\quad \textnormal{and}\quad U_j =  \bigcup_{i,k}W^{(ij)}_k.\] 
    We note that the ring homomorphisms corresponding to the morphisms of affine schemes
        \[\bigsqcup_{j,k} W^{(ij)}_k\to V_i\quad \textnormal{and}\quad \bigsqcup_{i,k} W^{(ij)}_k\to U_j\]
    are crisp since they are faithfully flat (cf. Lemma \ref{ff-crisp-rings}). In the commutative diagram of affine schemes
        \[\begin{tikzcd}\displaystyle
         \bigsqcup_{i,j,k} W^{(ij)}_k\arrow[r,"f_1"]\arrow[d,"f_2"] & \displaystyle\bigsqcup_i V_i \arrow[d,"g_1"]\\
        \displaystyle \bigsqcup_j U_j \arrow[r,"g_2"] & X,
        \end{tikzcd}\]
    the ring homomorphisms corresponding to $f_1$, $f_2$ are crisp by Lemma \ref{crisp-product-rings}, and the ring homomorphism $g_1$ is crisp by construction. Thus with Proposition \ref{permanence-crisp-rings}, the ring homomorphism corresponding to $g_2$ is crisp as wished.\\
    Thus so far, we have shown the equivalences $\textnormal{(i)}\Leftrightarrow\textnormal{(ii)}$ and $\textnormal{(iii)}\Leftrightarrow\textnormal{(iv)}$. It is clear that the former complex of statements implies the latter, so we now show the implication $\textnormal{(iv)}\Rightarrow \textnormal{(i)}$. We take an auxiliary step and show that if (iv) is true, then for every open affine covering $(W_i)_{i\in I}$ of $Y$ and for all $i$, there is a quasi-compact open $A_i\subseteq f^{-1}(W_i)$ with $f(A_i) = W_i$ such that there is a finite open affine covering $A_{i,1},\dots, A_{i,r_i}$ of $A_i$ for which the ring homomorphism corresponding to the morphism of affine schemes
            \[\bigsqcup_{k=1} ^{r_i} A_{i,k}\longrightarrow W_i\]
        is crisp.
    
    Thus, let us assume (iv): Let $(V_j)_{j\in J}$ be an open affine covering of $Y$ such that for each $V_j$, there is some quasi-compact open $U_j$ in $X$ with $U_j\subseteq f^{-1}(V_j)$ and $f(U_j) = V_j$ such that for every finite open affine covering $U_{j,1},\dots, U_{j,r_{j}}$ of $U_j$, the ring homomorphism corresponding to 
        \[\bigsqcup_{k=1} ^{r_j} U_{j,k}\to V_j\]
    is crisp. Now, choose some other open affine covering $(W_i)_{i\in I}$ of $Y$. Fix some $n\in I$. We know that $(V_j\cap W_n)_{j\in J}$ is an open covering of $W_n$. Since $W_n$ is affine and thus quasi-compact, $J$ is finite without loss of generality. We know that for any such $V_j$,  $(V_j\cap W_i)_{i\in I}$ is an open covering, and we can again assume $I$ to be finite, and to include $n$. Since $Y$ is quasi-separated, $V_j\cap W_i$ is quasi-compact and we find some finite open affine covering $(Z_l^{(ij)})_{l}$ of $V_j\cap W_i$. Then, we obtain the finite unions
        \[V_j = \bigcup_{i,l} Z_l^{(ij)}\quad \textnormal{and}\quad W_i= \bigcup_{j,l} Z_l^{(ij)},\]
    where the latter is in particular true for $W_n$. We want to show that for any $Z_l^{(ij)}$, there is some $A_l^{(ij)}\subseteq f^{-1}(Z_l^{(ij)})$ which is quasi-compact and maps surjectively onto $Z_l^{(ij)}$ such that there is a finite open affine covering $(B_m^{(ijl)})_{m}$ of $A_l^{(ij)}$ for which the ring homomorphism associated to 
        \[\bigsqcup_{m} B_m^{(ijl)} \to Z_l^{(ij)}\]
    is crisp. If we can find these sets, then setting 
        $$A_i = \bigcup_{j,l}A_l^{(ij)}$$ 
    yields a quasi-compact open subscheme of $X$ contained in $f^{-1}(W_i)$ mapping surjectively onto $W_i$ for which we also get: As
        \[\bigsqcup_m B_m^{(ijl)} \to Z_l^{(ij)}\]
    corresponds to a crisp ring homomorphism, so does
        \[\bigsqcup_{j,l,m} B_m^{(ijl)}\to \bigsqcup_{j,l} Z_l^{(ij)} \] 
    by Lemma \ref{crisp-product-rings}. Since
        \[\bigsqcup_{j,l} Z_l^{(ij)} \to W_i\]
    corresponds to a crisp morphism of affine schemes by faithful flatness of this map, 
    the ring homomorphism corresponding to 
        \[\bigsqcup_{m,j,l} B_m^{(ijl)}\to W_i\]
    is crisp as well. We are then done, since this guarantees for every $n\in I$ a quasi-compact $A_n$ in $f^{-1}(W_n)$ and $(B_m^{(njl)})_{m,j,l}$ as its finite open affine covering for which the definition of crispiness is satisfied.\\ Now, about the existence of all these sets: Consider for fixed indices the cartesian diagram
        \[\begin{tikzcd}
            U_{j,k}\times_{V_j}Z_l^{(ij)} \arrow[r]\arrow[d] & U_{j,k}\arrow[d]\\
            Z_l^{(ij)} \arrow[r] & V_j,
        \end{tikzcd}\]
    in which the bottom arrow is an open immersion.
    Setting $B^{(ijl)}_m =  U_{j,m}\times_{V_j}Z_l^{(ij)}$ and $A_l^{(ij)} = \bigcup_m B^{(ijl)}_m$, we obtain:
    \begin{itemize}
        \item All the $B^{(ijl)}_m$ are affine opens in $f^{-1}\left(Z_l^{(ij)}\right)$.
        \item They form a a finite open affine covering of $A_l^{(ij)}$, which is then open and quasi-compact as the finite union of open quasi-compacts.
        \item By construction, $A_l^{(ij)}$ is contained in $f^{-1}\left(Z_l^{(ij)}\right)$ and it maps surjectively onto $Z_l^{(ij)}$: We have the cartesian diagram 
            \begin{equation}\label{cart1}
                \begin{tikzcd}
                    \displaystyle \bigsqcup_m B^{(ijl)}_m \arrow[r]\arrow[d] & \displaystyle\bigsqcup_m U_{j,m}\arrow[d]\\
                    Z_l^{(ij)}\arrow[r] & V_j
                \end{tikzcd}
            \end{equation}
        in which the right vertical arrow is surjective, thus the left one is as well. This map factors through
            \[ \bigsqcup_m B^{(ijl)}_m \longrightarrow  \bigcup_m B^{(ijl)}_m = A_l^{(ij)}\longrightarrow Z_l^{(ij)}, \]
        with which the second arrow is surjective.
        \item In the cartesian Diagram \ref{cart1}, the right vertical arrow corresponds to a crisp ring homomorphism, thus the left one does as well by Proposition \ref{permanence-crisp-rings}.
    \end{itemize}
    These are all the requirements we posed such that the proof is concluded.
\end{proof}

\subsection{Permanence properties}

In this section, we will study some permanence properties of crisp morphisms.
\begin{Theorem}\label{permanence-crisp}
    Crisp morphisms of quasi-separated schemes satisfy the following properties:
    \begin{assertionlist}
        \item Being crisp is stable under composition.
        \item Being crisp is local on the target.
        \item Being crisp is stable under base change.
    \end{assertionlist}
\end{Theorem}

\begin{proof}
    (1)\quad Let $f\colon X\to Y$ and $g\colon Y\to Z$ be crisp morphisms of quasi-separated schemes. Their composition is locally quasi-compact surjective. For its crispness, we choose some open affine $W\subseteq Z$. Since $g$ is crisp, there is some quasi-compact open $V$ contained in $g^{-1}(W)$ mapping surjectively onto $W$ and finitely many affine opens $V_1,\dots, V_r$ covering $V$ such that the ring homomorphism corresponding to the morphism of affine schemes
        \[\bigsqcup_{i=1}^r V_i\to W\]
    is crisp. Since $f$ is crisp as well, we find for any $i$ a quasi-compact open $U_i$ contained in $f^{-1}(V_i)$ mapping surjectively onto $V_i$ and finitely many affine opens $U_{i,1},\dots, U_{i,s_i}$ covering $U_i$ such that the ring homomorphism corresponding to the morphism of affine schemes
        \[\bigsqcup_{j=1}^{s_i} U_{i,j}\to V_i\]
    is crisp for any $i$. With Lemma \ref{crisp-product-rings}, the ring homomorphism corresponding to the morphism of affine schemes
        \[\bigsqcup_{i,j} U_{i,j} \to \bigsqcup_{i} V_i\]
    is crisp as well. Then, the composition 
        \[\bigsqcup_{i.j} U_{i,j} \to W\]
    corresponds to a crisp ring homomorphism too. With $U = \bigcup_{i} U_i$, we have found some quasi-compact open in $(g\circ f)^{-1}(W)$ mapping surjectively onto $W$ which has a finite affine open covering $(U_{i,j})_{i,j}$ satisfying what is needed for the crispness of $g\circ f$. \\
    (2)\quad Let $f\colon X\to Y$ be a morphism of quasi-separated schemes and let $(V_i)_{i\in I}$ be an open covering of $Y$. We first note that being locally quasi-compact surjective is local on the target. Since any affine open subscheme of any of the $V_i$ is an affine open subscheme of $Y$, $f^{-1}(V_i)\to V_i$ is crisp for any $i$ whenever $f$ is. Conversely, if $f^{-1}(V_i)\to V_i$ is crisp for all $i$, we find for any $i$ an open affine covering $(U_j^{(i)})_{j\in J_i}$ of $V_i$ witnessing the crispness of $f^{-1}(V_i)\to V_i$. The union of all of these yields an open affine covering of $Y$ witnessing the crispness of $f$.\\
    (3)\quad Consider a cartesian diagram of quasi-separated schemes
        \[\begin{tikzcd}
            X'\arrow[r, "g'"]\arrow[d,"f'"] & X\arrow[d,"f"]\\
            S'\arrow[r,"g"] & S
        \end{tikzcd}\]
    where $f$ is crisp. We know that being locally quasi-compact surjective is stable under base change. Let $Z$ be an affine open in $S$. Consider its preimage in $S'$, which can be covered by affine opens $(V_i)_{i\in I}$. Since $f$ is crisp, there is some quasi-compact open $W$ contained in $f^{-1}(Z)$ mapping surjectively onto $Z$ such that there is some finite open affine covering $W_1,\dots,W_r$ of $W$ for which the ring homomorphism corresponding to the morphism of affine schemes
        \[\bigsqcup_{k=1}^r W_k\to Z\]
    is crisp. We now note:
    \begin{itemize}
        \item For any fixed $i$, the $W_k\times_Z V_i$ are finitely many open affines lying in $f'^{-1}(V_i)\subseteq X'$.
        \item Since the commutative diagram of affine schemes
        \[\begin{tikzcd}
            \displaystyle\bigsqcup_k W_k\times_Z V_i \arrow[d]\arrow[r] & \displaystyle\bigsqcup_k W_k\arrow[d]\\
            V_i\arrow[r] & Z
        \end{tikzcd}\]
        is cartesian, the ring homomorphism corresponding to the map 
            \[\bigsqcup_k W_k\times_Z V_i\to V_i\]
        is crisp for any $i$.
        \item The union $\bigcup_k W_k\times_Z V_i$ is quasi-compact and open and lies in $f'^{-1}(V_i)$. 
    \end{itemize}
    Repeating this argument for any affine open in $S$ yields an open affine covering of $S'$ witnessing the crispness of $f'$.
\end{proof}

\begin{Remark}
We refer to \cite[Section 11.10]{EGAIV}, \cite{Mes04a}, \cite{Pic13} and \cite{Mes04b} for similar statements about pure (i.e. universally schematically dominant) morphisms of schemes. Important ones, which we will need later on, are as follows:
\begin{itemize}
    \item Being pure is local on the target, stable under base chance and stable under composition.
    \item If a composition $g\circ f$ of morphisms of schemes is pure, then so is $g$.
    \item Any pure morphism of schemes is surjective.
\end{itemize}
\end{Remark}

\subsection{Crisp morphisms in context}
In this section, we will draw connections between crisp morphisms and other important classes of morphisms.

\begin{Proposition}\label{faithfully-flat-crisp}
    Any fpqc morphism of schemes is crisp. 
\end{Proposition}

\begin{proof}
    Let $f\colon X\to Y$ be fpqc, i.e. faithfully flat and locally quasi-compact surjective. Choose some open affine subscheme $V\subseteq Y$. Since $f$ is locally quasi-compact surjective, there is some quasi-compact open $U\subseteq f^{-1}(V)$ mapping surjectively onto $V$. We cover it by finitely many affine opens $U_1,\dots, U_r$ and we want to show that the ring homomorphism corresponding to the morphism of affine schemes
        \[\bigsqcup_{i=1}^r U_i\to V\]
    is crisp. If we can show that this morphism of affine schemes is faithfully flat, i.e. flat and surjective, the statement follows immediately. \\
    The surjectivity is immediate, given the fact that $U_1,\dots, U_r$ cover $U$ and that $U$ maps surjectively onto $V$. The morphism is flat since $U_i\to V$ is flat by definition, and flatness is local on the source.
\end{proof}

\begin{Proposition}\label{equivalent-on-affines}
    Let $f\colon \Spec B\to \Spec A$ be a morphism of affine schemes. Then, the following are equivalent:
    \begin{equivlist}
        \item The map $f$ is a crisp morphism of schemes.
        \item The map $f$ is a pure morphism of schemes, i.e. universally schematically dominant.
        \item The ring homomorphism $A\xrightarrow{\varphi_f} B$ is crisp in the sense of Definition \ref{crisp-ring-hom}.
    \end{equivlist}
\end{Proposition}

\begin{proof}
    The equivalence $\textnormal{(ii)}\Leftrightarrow \textnormal{(iii)}$ is due to \cite[Theorem 5.12]{Mes04a}.\\
    The implication $\textnormal{(iii)}\Rightarrow \textnormal{(i)}$ is by definition.\\
    For the implication $\textnormal{(i)}\Rightarrow\textnormal{(iii)}$, let $f\colon (\Spec B = X)\to (Y = \Spec A)$ be a crisp morphism of affine schemes. Then, there is some quasi-compact open $U\subset X$ and finitely many open affines $U_1,\dots, U_n$ covering $U$ such that the ring homomorphism corresponding to the morphism of affine schemes
        \[\bigsqcup_{i=1}^n U_i\to Y \]
    is crisp. This map factors through
        \[\bigsqcup_{i=1}^n U_i\to X\to Y,\]
    and by Proposition \ref{permanence-crisp-rings}, we get that the ring homomorphism $A\to B$ is crisp as well.
\end{proof}

\begin{Proposition}\label{pureqcqs-crisp}
    Let $f\colon X\to Y$ be a morphism of quasi-separated schemes. Then, $f$ is pure, i.e. universally schematically dominant, if it is crisp. If moreover, $f$ is quasi-compact, then the converse implication is true.
\end{Proposition}

\begin{proof} 
    Let $f$ be crisp. Since being crisp is stable under base change, it suffices to show that $f$ is schematically dominant, i.e. that the morphism of sheaves on $Y$
        \[\O_Y\to f_*\O_X\]
    is injective. This is equivalent to showing that for any $V\subseteq Y$ open affine, the ring homomorphism 
        \[\O_V(V)=\O_Y(V)\to f_*\O_X(V) = \O_{f^{-1}(V)}(f^{-1}(V))\]
    is injective. Since $f$ is crisp, there is some quasi-compact open $U$ contained in $f^{-1}(V)$ mapping surjectively onto $V$ and such that there is a finite open affine covering $U_1,\dots, U_n$ of $U$ making the ring homomorphism 
        \[\O_V(V)\to \O_{\bigsqcup_{i=1}^n U_i}\left(\bigsqcup_{i=1}^n U_i\right)\]
    crisp and thus injective.
    After pushing all structure sheaves forward to $V$ and taking global sections, the composition
        \[\bigsqcup_{i=1}^n U_i\to U\hookrightarrow f^{-1}(V)\to V\]
    induces the composition of ring homomorphisms
        \[\O_V(V)\to \O_{f^{-1}(V)}(f^{-1}(V))\to \O_U(U)\to \O_{\bigsqcup_{i=1}^n U_i}\left(\bigsqcup_{i=1}^n U_i\right).\]
    Using the cancellation property of injectivity, it now immediately follows that the first map, i.e.
        \[\O_V(V)\to \O_{f^{-1}(V)}(f^{-1}(V)),\]
    is injective as wished.\\
    Assume that conversely, $f$ is quasi-compact and schematically dominant. Then, $f$ is in particular locally quasi-compact surjective. Choose some open affine $V\subseteq Y$. Its preimage $f^{-1}(V)$ is quasi-compact. Choose any finite open affine covering $V_1,\dots, V_r$ of $f^{-1}(V)$. The morphism $f$ is crisp if the ring homomorphism corresponding to 
        \[\bigsqcup_{i=1}^r V_i\to V\]
    is crisp. According to Proposition \ref{equivalent-on-affines} right above, this is the case if and only if the map $\bigsqcup_{i=1}^r V_i\to V$ is universally schematically dominant. We know that 
    \begin{itemize}
        \item the map 
            \[\bigsqcup_{i=1}^r V_i\to f^{-1}(V)\]
        is faithfully flat and quasi-compact, thus universally schematically dominant (cf. \cite[Remark 3.13]{Mes04a})
        \item and the map $f^{-1}(V)\to V$ is universally schematically dominant since this property is local on the target.
    \end{itemize}
    Since 
        \[\bigsqcup_{i=1}^r V_i\to V\]
    factors through
        \[\bigsqcup_{i=1}^r V_i\to f^{-1}(V)\to V\]
    and being schematically dominant is stable under composition, we are done.
\end{proof}

\begin{Example}\label{sch-examples}
    \begin{enumerate}[(1)]
        \item As alluded to in the introduction, demanding a morphism of affine schemes $f\colon \Spec B\to \Spec A$ to be surjective and crisp on the stalks does not coincide with the ring-theoretic notion of crispness. Indeed, let $k$ be an algebraically closed field and let 
            \[f\colon \Spec k[X,Y]/(XY)\to \Spec k[X]\] 
        be the surjective map which projects the affine coordinate cross onto $\mathbb{A}^1_k$. The corresponding ring homomorphism 
            \[k[X]\to k[X,Y]/(XY)\]
        is crisp since it is split injective. However, the morphism of affine schemes $f$ is not crisp on the stalks: Choose $\mathfrak{q}=(Y-a)\subseteq  k[X,Y]/(XY)$ with $a\neq 0.$ The map $f^\#_\mathfrak{q}$ is given by
            \[k[X]_{(X)}\to \bigl(k[X,Y]/(XY)\bigr)_\mathfrak{q}\]
        and by the universal property of the quotient, it factors through
            \[\begin{tikzcd}[row sep =3em,column sep=-2em]
                k[X]_{(X)}\arrow[rr,"f^\#_\mathfrak{q}"]\arrow[dr,"\pi"]&&\bigl(k[X,Y]/(XY)\bigr)_\mathfrak{q}\\
                &k=k[X]_{(X)}/(X)k[X]_{(X)}\arrow[ur].&
            \end{tikzcd}\]
        Since $\pi$ is not injective, $f^\#_\mathfrak{q}$ cannot be injective either. In particular, it is not a crisp ring homomorphism.
        \item Consider the morphism of schemes 
            \[\bigsqcup_{p \textnormal{ prime}} \Spec \mathbb{F}_p\to \Spec \Z.\]
        Since 
            \[\Ker\left(\Z\to \prod_p\mathbb{F}_p\right)=\bigcap_p p\Z =0,\]
        this map is injective on global sections. It can, however, never be crisp! Any quasi-compact open in $\bigsqcup_p \Spec \mathbb{F}_p$ is of the form 
            \[\bigsqcup_{\textnormal{fin. many }p} \Spec \mathbb{F}_p.\]
        Taking any open affine covering of this yields the map 
            \[\bigsqcup_{\textnormal{fin. many }p} \Spec \mathbb{F}_p\to \Spec \Z,\]
        which on global sections never has trivial kernel.
        \item We can use the garbage principle (cf. Remark \ref{garbage}) to construct a morphism of schemes which is crisp but not fpqc. If $f\colon X\to Y$ is a morphism of schemes such that $f|_U$ is fpqc for some open subscheme $U\subsetneq X$, then $f|_U$ is crisp and thus $f$ is crisp. In general however, $f$ will not be fpqc. 
    \end{enumerate}
\end{Example}

\subsection{Excursion: subtrusive and submersive morphisms}

In this subsection, we will give some information on subtrusive and submersive morphisms of schemes. This will become important in the following chapter, where we will need the fact that any quasi-compact crisp morphism is universally subtrusive for some proofs.\\
Submersiveness and subtrusiveness are topological properties of morphisms of schemes that possess interesting (descent) properties. Firstly, any faithfully flat quasi-compact morphism of schemes is universally submersive, and this is actually what we need for the faithfully flat descent of some of the topological properties of morphisms of schemes. Secondly, universally subtrusive morphisms are effective descent morphisms in the fibered category of étale morphisms. This is just an appetiser, and we recommend \cite{Pic86}, \cite{Rydh10} and \cite{Pic13} for further reading.

We give some topological context for the definition of these types of morphisms:

\begin{Reminder}
    We will concern ourselves with the following topologies on a scheme $X$ (cf. \cite[p. 4f.]{Rydh10}, \cite[Definition 10.39]{GW10}):
    \begin{bulletlist}
        \item Zariski topology
        \item Constructible topology: also known as the patch topology. The closed subsets are the pro-constructible subsets, and the open subsets are the ind-constructible subsets, meaning subsets that are locally an intersection or a union of constructible subsets, respectively. Constructible subsets are given by finite unions of sets $U\cap V^{C}$, where $U$ and $V$ are open and retrocompact subsets of $X$. A subset $W$ of $X$ is called retrocompact if $W\cap Z$ is quasi-compact for every quasi-compact open subset $Z$ of $X$. If $X$ is qcqs, these notions simplify: Then, a subset of $X$ is constructible if it is a finite union of subsets $U\cap (X\setminus V)$ where $U$ and $V$ are open quasi-compact, and the pro-constructible subsets are of the form $f(\Spec A)$, where $f\colon \Spec A\to X$ is a morphism of qcqs schemes.
        \item S-topology: Here, the S stands for specialisation w.r.t. the Zariski topology and the open sets are those which are stable under generisation w.r.t. the Zariski topology.
    \end{bulletlist}
\end{Reminder}

\begin{Definition}\label{submersive}
    A surjective morphism of topological spaces $f\colon X\to Y$ is called \textit{submersive} if the topology on $Y$ is the quotient topology with respect to $f$, i.e. a subset $Z$ of $Y$ is closed if and only if $f^{-1}(Z)$ is closed.\\
    A surjective morphism of schemes $f\colon X\to Y$ is called \textit{submersive} if the underlying morphism of topological spaces is submersive, cf. \cite[p. 5]{Rydh10}.
\end{Definition}

\begin{Definition}\label{subtrusive}
    A morphism of schemes $f\colon X\to Y$ is called \textit{subtrusive} if the following two conditions hold (cf. \cite[Definition 1]{Pic13}):
    \begin{definitionlist}
        \item Every ordered pair $y\prec y'$ of points (w.r.t. the Zariski topology) in $Y$ lifts to an ordered pair of points $x\prec x'$ (w.r.t. the Zariski topology) in X, i.e. for every ordered pair $y\prec y'$ of points in $Y$ there is an ordered pair of points $x\prec x'$ in X such that $f(x) = y$ and $f(x') = y'.$
        \item The morphism $f$ is submersive in the constructible topology.
    \end{definitionlist}
\end{Definition}

\begin{Remark}\label{submersive in different topologies}
    One can show that a subset of a scheme is open (resp. closed) in the Zariski topology if and only if it is open (resp. closed) both in the constructible and in the S-topology (cf. \cite[Corollary 1.5]{Rydh10}). It follows that any morphism is submersive in the Zariski topology if it is submersive in the constructible and in the S-topology.
\end{Remark}

This leads to the following statement:

\begin{Lemma}\label{subtrusive-submersive}
    Any subtrusive morphism of schemes is submersive.
\end{Lemma}

\begin{proof}
    By Remark \ref{submersive in different topologies}, we only need to check why the morphism is submersive in the S-topology. This is Proposition 2.1 in \cite{Rydh10}.
\end{proof}

\begin{Example}
    Note that, in general, the converse does not hold, not even for qcqs schemes! A counterexample can be found in Section 8 of \cite{Pic86}.
    One can show that the two properties are equivalent over locally Noetherian schemes (cf. \cite[p. 2]{Rydh10}).
\end{Example}

The following proposition explains why we have introduced these notions:

\begin{Proposition}\label{crisp-subtrusive}
    Any quasi-compact crisp morphism of schemes is universally subtrusive.
\end{Proposition}

\begin{Lemma}\label{val-lemma}
    A quasi-compact morphism of schemes $f\colon X\to Y$ is universally subtrusive if and only if for every morphism of schemes $\Spec V\to Y$, where $V$ is a valuation domain, the induced morphism of schemes
    \[X\times_Y \Spec V \to \Spec  V \]
    is pure.
\end{Lemma}

\begin{proof}
    Cf. \cite[Theorem 50]{Pic13}.
\end{proof}

\begin{proof}[Proof of Proposition~\ref{crisp-subtrusive}]
    If $f$ is quasi-compact and crisp, then so is the morphism $X\times_Y \Spec V \to \Spec  V$ for any valuation domain $V$ since being crisp and quasi-compact is stable under base change. Then, $X\times_Y \Spec V \to \Spec  V$ is pure by Proposition \ref{pureqcqs-crisp} and thus universally subtrusive by Lemma \ref{val-lemma}. 
\end{proof}
    \section{The crisp topology}\label{chap::four}

Let $S$ be a scheme and consider the slice category (Sch/$S$) of schemes over $S$. In the following, we will use the notion of a \textit{Grothendieck topology} as defined in \cite{FA05}, meaning a \textit{pretopology} in the sense of FGA. In the subsections below, we will introduce a new Grothendieck topology on (Sch/$S$), the \textit{crisp} topology, and also study some of its properties, most notably properties of schemes which are local on the target w.r.t. our topology. Furthermore, we will show that it is subcanonical.

We will construct the crisp Grothendieck topology adhering to the structure of \cite{FA05}, where, in Chapter 2.3, the fpqc topology is introduced. 
 
\begin{Reminder}
    The \textit{fpqc topology} on (Sch/$S$) is the topology where the coverings $\{U_i\to U\}_{i\in I}$ are collections of morphisms such that the induced morphism
        \[\bigsqcup_i U_i\to U\]
    is faithfully flat and locally quasi-compact surjective.
\end{Reminder}

The fpqc topology induces a well-behaved site in the sense that it is both finer than the Zariski topology and subcanonical. Drawing back to the ring- and scheme-theoretic relations between faithfully flat and crisp morphisms, the following definition does not come as a surprise:

\begin{Definition}\label{crisp-top}
    The \textit{crisp topology} on (Sch/$S$) is the topology in which the coverings $\{U_i\to U\}_{i\in I}$ are collections of morphisms such that the induced morphism
        \[\bigsqcup_i U_i\to U\]
    is crisp.
\end{Definition}

\begin{proof}[Proof that Definition \ref{crisp-top} indeed gives a Grothendieck topology]
    Firstly, any isomorphism of schemes is crisp. Secondly, being crisp is stable under composition and base change and can moreover be checked locally on the target.
\end{proof}

\begin{Remark}
    Since the fpqc topology is finer than the Zariski topology and every faithfully flat morphism of schemes is crisp, the crisp topology is finer than the fpqc and the Zariski topology. Below we will see that the crisp topology is subcanonical.
\end{Remark}

\subsection{Crisp-local properties}

In this subsection, we want to determine which properties of schemes are crisp-local on the target: Let $f\colon X\to Y$ be a morphism of $S$-schemes and $\{Y_i\to Y\}_{i\in I}$ a crisp covering. Let \textbf{P} be a property of morphisms of schemes and assume that for all $i$, the projection $Y_i\times_Y X\to Y_i$ has property \textbf{P}. We are interested in whether the morphism $f\colon X\to Y$ also has property \textbf{P}. We note that this is true if
\begin{definitionlist}
    \item \textbf{P} is local on the target in the Zariski topology and
    \item \textbf{P} descends along morphisms of schemes which are crisp and quasi-compact.
\end{definitionlist}
This can be seen as follows: As \textbf{P} is local on the target in the Zariski topology, the morphism
    \[\bigsqcup_i Y_i\times_Y X\to \bigsqcup_i Y_i\]
also has property \textbf{P}. Let us now consider the cartesian diagram of $S$-schemes
    \[\begin{tikzcd}
        \displaystyle\bigsqcup_i Y_i\times_Y X \arrow[d]\arrow[r] & X\arrow[d, "f"]\\
        \displaystyle\bigsqcup_i Y_i\arrow[r, "\pi"] & Y,
    \end{tikzcd}\]
where the map $\pi$ is crisp and locally quasi-compact surjective. By property (iii) from the definition of crispness (Definition \ref{crisp-equivalences}), there is an open affine covering $(V_j)_{j\in I}$ of $Y$ such that each $V_j$ is the image of some quasi-compact open subset $W_j$ of $\pi^{-1}(V_j)$ and such that $W_j$ witnesses the crispness of $\pi^{-1}(V_j)\to V_j$. Then, $\pi_{|W_j}$ is crisp and quasi-compact. For every $j$, we have an induced cartesian diagram of $S$-schemes
    \[\begin{tikzcd}
        \displaystyle \bigsqcup_i f^{-1}(V_j)\times_{V_j} W_j \arrow[d]\arrow[r] & f^{-1}(V_j)\arrow[d, "f\vert_{f^{-1}(V_j)}"]\\
        W_j\arrow[r, "\pi\vert_{W_j}"] & V_j.
    \end{tikzcd}\]
Since \textbf{P} is local on the target in the Zariski topology, we may deduce two things: Firstly, the left vertical arrow of the diagram has property \textbf{P}. Secondly, showing that $f\vert_{f^{-1}(V_j)}$ has \textbf{P} for any $j$ suffices for $f$ to have \textbf{P}. As $\pi\vert_{W_j}$ is crisp and quasi-compact, we now know that $f$ has \textbf{P} whenever it descends along crisp and quasi-compact maps.

\begin{Theorem}\label{crispt-descent}
    Let $f\colon X\to Y$ be a morphism of schemes and $g\colon Y'\to Y$ a crisp and quasi-compact morphism of schemes. If $f'\colon (X'\coloneqq X\times_Y Y')\to Y'$ has \textbf{P}, then $f$ has \textbf{P}, where \textbf{P} is one of the following properties:
    \begin{assertionlist}
        \item surjective
        \item injective
        \item bijective
        \item a morphism with set-theoretically finite fibres
        \item open
        \item closed
        \item a homeomorphism
        \item quasi-compact
        \item quasi-separated
        \item separated
        \item (locally) of finite type
        \item (locally) of finite presentation
        \item smooth
        \item unramified
        \item étale
        \item an isomorphism
        \item a monomorphism
        \item an open immersion
        \item proper
        \item quasi-finite
    \end{assertionlist}
\end{Theorem}

\begin{proof}
    As $f$ is a crisp morphism of schemes and thus surjective, the statements regarding Properties (1) through (4) follow immediately by \cite[Proposition 14.50]{GW10}. 

    Properties (5) through (10) descend along morphisms which are both quasi-compact and universally submersive (cf. \cite[Proposition 14.51 and Remark 14.52 (1)]{GW10}). Thus, they descend along crisp and quasi-compact morphisms (see Lemma \ref{subtrusive-submersive} and Proposition \ref{crisp-subtrusive}).

    Properties (11) through (15) can be reduced to the corresponding statements for rings (cf. Propositions \ref{descent-algebras} and \ref{descent-etale-etc}): Let \textbf{P} be one of the following: locally of finite type, locally of finite presentation, smooth, unramified or étale. Then, $f$ has \textbf{P} if and only if for every open affine $U = \Spec B$ in $X$ and every open affine $V = \Spec A$ in $Y$ with $f(U)\subseteq V$, the ring homomorphism $A\to B$ has \textbf{P}. Since $g$ is crisp, we now obtain a finite open affine covering $W_1,\dots,W_n$ of a quasi-compact open $W\subseteq g^{-1}(V)$ with $g(W) = V$ and a cartesian diagram
        \[\begin{tikzcd}
            \displaystyle\bigsqcup_{i=1}^n W_i\times_{V} U\arrow[rr]\arrow[d] && U\arrow[d]\\
        \displaystyle\bigsqcup_{i=1}^n W_i\arrow[r] & W\arrow[r] & V
        \end{tikzcd}\]
    with affine corners and a crisp morphism at the bottom.
    As \textbf{P} is local on the target and stable under base change, the morphism
        \[\bigsqcup_{i=1}^n W_i\times_{V} U\to \bigsqcup_{i=1}^n W_i \]
    has \textbf{P}. Thus, $U\to V$ has \textbf{P} as wished, using the corresponding descent statement for rings. The properties being of finite type and of finite presentation also descend because of the descent of Properties (8) and (9).

    For the descent of isomorphisms (Property (16)), we note that by descent of the Properties (7) through (9), $f$ is a (universal) homeomorphism and qcqs whenever $f'$ is an isomorphism. It thus only remains to show that the sheaf morphism 
        \[\O_Y\to f_*\O_X\]
    is an isomorphism. We now use a similar strategy as for Properties (11) through (15): Being an isomorphism can be checked locally on affines in $Y$ for $\O_Y\to f_*\O_X$, so we may assume that $Y$ is affine. Then, there is a finite open affine covering $W_1,\dots,W_n$ of a quasi-compact open $W\subseteq Y'$ with $g(W) = Y$ such that in the cartesian rectangle
        \[\begin{tikzcd}
            \displaystyle\bigsqcup_{i=1}^n W_i\times_Y X\arrow[d,"f'''"]\arrow[r] &f'^{-1}(W)\arrow[r,hook]\arrow[d,"f''"]&X'\arrow[r]\arrow[d, "f'"]  & X\arrow[d, "f"]\\
            \displaystyle\bigsqcup_{i=1}^n W_i\arrow[r] & W\arrow[r, hook]& Y'\arrow[r, "g"] & Y,
        \end{tikzcd}\]
   the composition at the bottom is crisp. As $f'$ is an isomorphism, $f'''$ is as well. If the outer rectangle is a base change diagram of affine schemes, we can deduce the wished statement. We may indeed assume this: Since $f$ is a homeomorphism, it is affine, and because $Y$ is affine, $X$ is now as well.
   
    The argumentation for the descent of monomorphisms and open immersions (Properties (17) and (18)) can be taken ad verbatim from \cite[Proposition 14.53 (3)]{GW10}, which may be applied because any crisp map is submersive.

    The descent of proper morphisms (property (19)) is due to the fact that being proper just means being separated, universally closed and of finite type.

    The descent of quasi-finite morphisms (property (20)) is analogously due to the fact that being quasi-finite just means being of finite type with set-theoretically finite fibres.
\end{proof}

\begin{Corollary}\label{crisp-loc-props}
    Let $f\colon X\to Y$ be a morphism of $S$-schemes and $\{Y_i\to Y\}_{i\in I}$ a crisp covering. Let \textbf{P} be a property of morphisms of schemes as above and assume that for all $i$, the projection $Y_i\times_Y X\to Y_i$ has property \textbf{P}. Then, $f$ has property \textbf{P} as well.
\end{Corollary}

\begin{proof}
    This follows immediately from Theorem \ref{crispt-descent}, using that any of these properties is Zariski-local on the target.
\end{proof}

\subsection{The crisp topology is subcanonical}

\begin{Theorem}\label{subcanonical}
    The crisp topology on (Sch/$S$) is subcanonical, i.e. every representable functor on (Sch/$S$) is a sheaf with respect to the crisp topology.
\end{Theorem}

We strictly follow the proof of the corresponding Theorem 2.55 for the fpqc topology in \cite{FA05} and start off with a lemma which is well known for other Grothendieck topologies:

\begin{Lemma}\label{reducetoonecover}
    Let  $F\colon \textnormal{(Sch/$S$)}^{\mathrm{op}}\to (\mathrm{Set})$ be a functor. If $F$ satisfies
    \begin{definitionlist}
        \item F is a sheaf in the Zariski topology
        \item and whenever $V\to U$ is a crisp morphism of affine $S$-schemes, the diagram 
            \[\begin{tikzcd}
                F(U)\arrow[r] &F(V)\arrow[r,yshift=2pt]\arrow[r,yshift=-2pt] &F(V\times_U V)
            \end{tikzcd}\]
        is an equalizer,
    \end{definitionlist}
    then $F$ is a sheaf in the crisp topology.
\end{Lemma}

\begin{proof}
    The proof can be found in \cite[proof of Lemma 2.60]{FA05}, substituting ``fpqc'' with ``crisp'', which does not change the argumentation.
\end{proof}

We need another technical lemma: 

\begin{Lemma}\label{reducetoaffine}
    Let $A$ and $B$ be $R$-algebras. Let $f\colon A\to B$ be a crisp ring homomorphism. Then, the sequence
        \[0\longrightarrow A\overset{f}{\longrightarrow} B\overset{e_1-e_2}{\longrightarrow}B\otimes_A B,\]
    where $e_1(b) = b\otimes 1$ and $e_2(b) = 1\otimes b$, is exact.
\end{Lemma}

\begin{proof}
    The proof in \cite{FA05} for the corresponding Lemma 2.61 in the fpqc case shows that this exactness is true after some faithfully flat base change. This exactness also descends along crisp ring homomorphisms by Theorem \ref{crisp-equivalences}.
\end{proof}

\begin{proof}[Proof of Theorem \ref{subcanonical}]
    Let $X$ be an $S$-scheme. We want to show that $h_X$ is  a sheaf in the crisp topology. It is a Zariski sheaf and so using Lemma \ref{reducetoonecover}, it suffices to check the sheaf condition for $h_X$ on crisp coverings by a single scheme (condition (b)).\\
    Let us first assume that $X$ is an affine scheme. Then, condition (b) from Lemma \ref{reducetoonecover} becomes Lemma \ref{reducetoaffine}. To deduce the general case, one only uses the fact that fpqc and crisp morphisms of affine schemes have the common property of being submersive and again follows the proof in \cite{FA05}. 
\end{proof}
    \printbibliography

    \textsc{Saskia Kern, Technische Universität Darmstadt, Fachbereich Mathematik FB04, Schlossgartenstraße 7, 64289 Darmstadt, Germany}
    
    \textit{Email address:} \texttt{skern@mathematik.tu-darmstadt.de}
\end{document}